\documentclass[a4paper,12pt]{amsart}
\usepackage{amsmath,amsfonts,amsthm,amssymb,graphicx,stmaryrd,exscale,xifthen}
\usepackage[pdftex,
	pdfauthor=Thomas Bloom,
	colorlinks
]{hyperref}
\newtheorem{theorem}{Theorem}[section]
\newtheorem{lemma}{Lemma}[section]

\newtheorem{corollary}{Corollary}[section]

\theoremstyle{definition}

\newcommand{\supp}{\mbox{supp}}
\newcommand{\rst}[1]{\ensuremath{{\mathbin\upharpoonright}\raise-.5ex\hbox{$#1$}}}  
\newcommand{\bbf}{\mathbb{F}}
\newcommand{\bbn}{\mathbb{N}}
\newcommand*{\bbe}{\ifinner\mathbb{E}\else\mathop{\vcenter{\hbox{\huge$\mathbb{E}$}}}\fi }
\newcommand{\bbz}{\mathbb{Z}}

\newcommand{\abs}[1]{\left\lvert #1\right\rvert}
\newcommand{\brac}[1]{\left( #1\right)}
\newcommand{\norm}[1]{\left\lVert #1\right\rVert}
\newcommand{\bbc}{\mathbb{C}}
\numberwithin{equation}{section}

\begin{document}
 \title{Translation invariant equations and the method of Sanders}
\author{Thomas F. Bloom}
\date{\today}
\address{Thomas Bloom\\Department of Mathematics\\University of Bristol\\
University Walk\\Clifton\\ Bristol BS8 1TW\\United Kingdom}
\email{matfb@bristol.ac.uk}
\thanks{The author is supported by an EPSRC doctoral training grant.}
\begin{abstract}
We extend the recent improvement of Roth's theorem on three term arithmetic progressions by Sanders to obtain similar results for the problem of locating non-trivial solutions to translation invariant linear equations in many variables in both $\bbz/N\bbz$ and $\bbf_q[t]$. 
\end{abstract}
\maketitle
\section{Introduction}
This paper concerns solutions to a linear equation in $s\geq 3$ variables,
\begin{equation}\label{main}
 c_1x_1+\cdots+c_sx_s=0\,\,\,\,\,\,\,c_i\neq0,
\end{equation}
working inside some fixed ring $R$, and we restrict our attention to those equations with coefficients satisfying $c_1+\cdots+c_s=0$. These are often referred to as translation invariant systems, since it follows that if $(x_1,\hdots, x_s)\in R^s$ is a solution then so is $(x_1+x,\hdots,x_s+x)$ for any $x\in R$. Following Ruzsa \cite{ruzsa} we define the {\em genus} of an equation of the shape \eqref{main} to be the largest number $m$ with the property that there is a partition of $\{1,\hdots,s\}$ into $m$ disjoint nonempty sets $T_j$ where $\sum_{i\in T_j}c_i=0$. Note that the genus is well defined and positive, as a consequence of translation invariance. Given an equation \eqref{main} of genus $m$ and a finite set $A$ one obtains $\lvert A\rvert^m$ solutions $\mathbf{x}\in A^s$ by setting $x_i=x_{i'}$ whenever $i,i'\in T_j$. We call such solutions trivial, and seek an upper bound on the size of sets which contain only trivial solutions to \eqref{main}. In this paper we generalise a recent result of Sanders which gives the best known bound in the case with $s=3$ and $R=\bbz/N\bbz$ to obtain quantitative results of comparable quality for any $s\geq 3$ in circumstances where $R$ is either $\bbz/N\bbz$ or the polynomial ring $\bbf_q[t]$. Despite the analogies between these rings having been well-explored throughout most of number theory, the problems of additive combinatorics have been little studied in the $\bbf_q[t]$ setting, and we hope that this paper will encourage others to obtain further results of this nature in polynomial rings. We also wish to promote the philosophy that $\bbf_q[t]$ is a useful model case for such problems, acting as a `halfway house' between the currently popular finite field model case $\bbf_p^N$ and $\bbz/N\bbz$; technically simpler than the latter but capturing more of the interesting behaviour of the integer case than the former.

The problem of finding large sets with no non-trivial solutions to \eqref{main} may be posed in any ring (or indeed any module), but since the work of Roth (see \cite{roth,roth2}) it has received most attention in the integers $\bbz$, or rather finite truncations $\{1,\hdots, N\}$ that may conveniently be viewed as the cyclic group $\bbz/N\bbz$. In the special case $s=3$ and $(c_1,c_2,c_3)=(1,-2,1)$ this is equivalent to finding a three term arithmetic progression $a,a+d,a+2d$, with non-trivial solutions satisfying $d\neq0$. If $s\geq 3$ and $\mathbf{c}\in(\bbz/N\bbz)^s$ is such that \eqref{main} has genus $m\geq 1$ then we define $r(N)=r_{s,\mathbf{c}}(N)$ to be the cardinality of the largest subset of $\bbz/N\bbz$ which contains no non-trivial solutions to \eqref{main}. Roth showed that $r(N)\ll_{s,\mathbf{c}} N/\log\log N$. This was improved for $s=3$ by Heath-Brown \cite{heath-brown} and Szemer\'{e}di \cite{szemeredi} to $r(N)\ll N/\log^c N$ for some absolute constant $c>0$, the value of which was subsequently improved by Bourgain (see \cite{bourgain1,bourgain2}), first to $c=1/2-o(1)$ and then to $c=2/3-o(1)$. Sanders further improved this to $c=3/4-o(1)$ in \cite{sanders2} before the recent breakthrough result of \cite{sanders1},
\begin{equation}\label{sanders}r(N)\ll N\frac{(\log\log N)^5}{\log N}.\end{equation}

All of the results mentioned above were obtained only for the special case of three term progressions when $\mathbf{c}=(1,-2,1)$, but it is straightforward to generalise their methods for any translation invariant equation of the form \eqref{main} with $s=3$.

We expect that as the number of variables $s$ increases non-trivial solutions to \eqref{main} should become easier to find, and hence we should obtain improved bounds for $r(N)$. The first result of this paper confirms this expectation, generalising the result of Sanders to handle arbitrary $s\geq 3$. 
\begin{theorem}\label{mainthm1}
Let $s\geq 3$, and suppose $\mathbf{c}\in\bbz^s$ is such that the equation \eqref{main} has genus $m\geq 1$. Then
\[r_{s,\mathbf{c}}(N)\ll_{s,\mathbf{c}}N\left(\frac{(\log\log N)^5}{\log N}\right)^{s-2}.\]
\end{theorem}

We remark that the implicit constant in Theorem \ref{mainthm1} in fact depends only on $\ell=\max_{1\leq i\leq s}\lvert c_i\rvert$.
In \cite{ruzsa} Ruzsa showed that if the equation \eqref{main} has genus $m$ then $r(N)\ll N^{1/m}$, which is far superior to Theorem \ref{mainthm1} whenever $m\geq 2$. The power of the methods used to obtain the bounds listed above, which originated in Roth \cite{roth}, is that they give a non-trivial result in the most difficult case $m=1$. When $s\geq 6$, recent work of Schoen and Shkredov \cite{schoen}, building on other work of Sanders, shows that for some absolute positive constants $C$ and $c$ we have the near-optimal bound $r(N)\ll N\exp(-C\log^cN)$. Theorem \ref{mainthm1}, however, gives the sharpest known bounds for $s=4$ and $s=5$. 

By partial summation we obtain the following corollary.
\begin{corollary}\label{erdos}
 If $A\subset \bbn$ satisfies $\sum_{a\in A}a^{-1}=\infty$ then $A$ contains infinitely many non-trivial solutions to every translation invariant equation of the form \eqref{main} having $s\geq 4$ variables.
\end{corollary}
For comparison, a conjecture of Erd\H{o}s asserts that if $\sum_{a\in A}a^{-1}=\infty$ then $A$ must contain infinitely many arithmetic progressions of length $k$ for any $k\in\bbn$. In the case $k=3$, this would follow were Corollary \ref{erdos} to hold with $s=3$, which would in turn follow if one could improve the $(\log\log N)^5$ factor in Theorem \ref{mainthm1} to $(\log\log N)^{-2}$, for example.

Another popular setting for the problem of finding solutions to the equation \eqref{main} is $\bbf_p^N$, an $N$-dimensional vector space over the finite field $\bbf_p$. Here we have coefficients $c_i\in \bbf_p$ and we define $r_{s,\mathbf{c}}(N)$ to be the cardinality of the largest subset of $\bbf_p^N$ that contains only trivial solutions to the equation \eqref{main}. The best known bounds here are superior to those in $\bbz/N\bbz$. Thus, for example, the simple method of Meshulam \cite{meshulam} yields $r(N)\ll p^N/N$, a conclusion comparable to $r(N)\ll N/\log N$ in the integer case. Furthermore, Bateman and Katz \cite{batemankatz} have recently improved this upper bound to show that for some absolute constant $\epsilon>0$, one has $r(N)\ll p^N/N^{1+\epsilon}$. The investigation of such problems in $\bbf_p^N$ has recently become popular since it is often conceptually easier to prove results in $\bbf_p^N$, taking advantage of the vector space structure, and then `translate' the methods to the more difficult case $\bbz/N\bbz$ (see \cite{green} for a comprehensive discussion of this technique). 

One setting that has received comparatively little attention is $\bbf_q[t]$, the ring of polynomials over a finite field; as in the case of $\bbz$ it is more convenient to deal with finite truncations of this infinite ring, so we shall work in the additive subgroup of polynomials with degree strictly less than $N$, which we shall denote by $G_N$. There are many well-known analogies between number theory in $\bbz$ and in $\bbf_q[t]$, and theorems in one often have a natural counterpart in the other. We should expect, in particular, some result analogous to \eqref{sanders} to hold for solutions to the equation \eqref{main} where $\mathbf{c}\in\bbf_q[t]^s$. Perhaps part of the reason why this problem has been largely overlooked is that $G_N$ is additively isomorphic to $\bbf_q^N$, since both are $N$-dimensional vector spaces over $\bbf_q$. Hence in the classical case, with $s=3$ and $\mathbf{c}=(1,-2,1)$, the strong bounds of the $\bbf_q^N$ setting are available. In general, such bounds are available whenever $\mathbf{c}\in\bbf_p^s$, where $p$ is the characteristic of $\bbf_q$. In the $\bbf_q[t]$ setting, however, this restricts all coefficients to have degree zero, and hence trivial size under the usual valuation $\lvert a\rvert=q^{\deg a}$. Viewed in this way it is less surprising that we can obtain such impressive bounds with such a strong restriction on the coefficients of our linear equation. Note in particular that such issues of triviality do not arise in the $\bbz/N\bbz$ case, since $s\geq 3$ and translation invariance forces at least one of the coefficients of the equation \eqref{main} to have non-trivial size.

When $s\geq 3$ and $\mathbf{c}\in\bbf_q[t]^s$, denote by $r(N)$ the size of the largest subset of $G_N$ which contains no non-trivial solutions to the equation \eqref{main}. The sharpest bound for $r(N)$ currently available in the case where $\mathbf{c}\in\bbf_q^s$ is due to Liu and Spencer \cite{liuspencer} who showed that $r(N)\ll q^N/N^{s-2}$. The second result of this paper adapts the method of Sanders to give a comparable bound even when the coefficients $c_i$ have large degree. As above, this result gives a bound which improves as $s$ increases.

\begin{theorem}\label{mainthm2}
Let $s\geq 3$, and suppose $\mathbf{c}\in\bbf_q[t]^s$ is such that the equation \eqref{main} has genus $m\geq 1$. Then
\[r_{s,\mathbf{c}}(N)\ll_{s,\ell}q^N\left(\frac{(\log N)^4}{N}\right)^{s-2},\]
where $\ell=\max(\deg c_i)$.
\end{theorem}
Note here that the implicit constant depends only on the highest degree of the coefficients $c_i$. It also depends on $q$, the size of the constant field. Throughout this paper, this will be considered fixed as will $s$ and $\mathbf{c}$, and all constants may depend on these.

It is not difficult to get some kind of quantitative bound for $r(N)$ even when the coefficients have large degree, and indeed the original method of Roth can be adapted with a little effort to obtain $r(N)\ll q^N/\log N$. Theorem \ref{mainthm2} goes beyond this to show that we can obtain a quantitative bound comparable to the best known in $\bbz/N\bbz$. Indeed, in $\bbf_q[t]$ we are able to do slightly better (by a factor of $\log N$) than the analogue of \eqref{sanders}. This improvement is a consequence of the vector space structure of $\bbf_q[t]$ which we are able to exploit even when the coefficients have large degree. 

We remark that the methods of \cite{ruzsa} are easily adapted to the $\bbf_q[t]$ setting to give $r(N)\ll q^{N/m}$ whenever \eqref{main} has genus $m$, a conclusion sharper than that of Theorem \ref{mainthm2} whenever $m\geq 2$. The strength of Theorem \ref{mainthm2} is that it applies when $m=1$, which appears to be the most difficult case. 

Since this paper was originally written, we have been made aware of recent similar results by Liu and Zhao \cite{liuzhao}. By adapting the methods of \cite{bourgain1} and \cite{bourgain2} to the $\bbf_q[t]$ setting they show, under the same hypotheses as Theorem \ref{mainthm2}, that
\[r_{s,\mathbf{c}}(N)\ll_{s,\ell}q^N\brac{\frac{(\log N)^2}{N}}^{\frac{2(s-2)^2}{4s-9}}.\]
This is weaker than Theorem \ref{mainthm2} for all $s\geq 3$, as it does not use the new techniques introduced in \cite{sanders1}. We note, however, that the mechanisms they use to translate ideas from the $\bbz/N\bbz$ setting of \cite{bourgain1} and \cite{bourgain2} to $\bbf_q[t]$ are very similar to those introduced in this paper. In particular, they also give the proper analogue of Bohr sets in the $\bbf_q[t]$ setting.

The main ideas and techniques used to establish Theorems \ref{mainthm1} and \ref{mainthm2} are those of the original argument in \cite{sanders1}, but we are able to make many technical simplifications in $\bbf_q[t]$, largely thanks to the fact the analogue of `Bohr sets' are closed under addition. Indeed, we hope that this paper may also serve as an exposition of the methods in \cite{sanders1}, for in $\bbf_q[t]$ the fundamental ideas are less obscured by technical difficulties. We further hope that this paper will indicate how to translate many results of additive combinatorics to an analogous result in $\bbf_q[t]$, which we believe to be a technically simpler model for such results.

In recent years there has been a focus on obtaining such results in $\bbf_p^N$ and then translating the methods used so as to obtain a result in $\bbz/N\bbz$. We believe that the translation from $\bbf_p^N$ to $\bbf_q[t]$ and thence to $\bbz/N\bbz$ is simpler to perform and more intuitive. We are confident that this new method will stimulate further progress in other problems of additive combinatorics.

Thus, for example, the approach taken in this paper is to first establish Theorem \ref{mainthm2}, the $\bbf_q[t]$ case, and then sketch how to adapt the proof to obtain Theorem \ref{mainthm1}, the $\bbz/N\bbz$ case. In Section \ref{notation} we explain the notation used in the rest of the paper and prove two useful lemmas. Sections \ref{kk} and \ref{cs} provide the two tools used in the approach of Sanders \cite{sanders1}, namely a combinatorial transformation of sets inspired by the work of Katz and Koester \cite{katzkoester} and an application of the useful probabilistic theorem due to Croot and Sisask \cite{crootsisask}. Section \ref{disection} combines these tools in the `density increment' strategy so as to conclude the proof of Theorem \ref{mainthm2}. Section \ref{integercase} then sketches a proof of Theorem \ref{mainthm1}. Here we are brief, since the proofs are the same as those for Theorem \ref{mainthm2} except for routine technical changes.

\section{Notation and preliminaries}\label{notation}
Throughout the rest of this paper, let $N$ be some fixed large integer and $\bbf_q$ be the finite field with $q$ elements. We shall write $G=G_N=\{ x\in\bbf_q[t] : \deg x<N\}$, and note in particular that $G$ is an $N$-dimensional vector space over $\bbf_q$. We will frequently use the expectation notation
\[\bbe_{x\in X}f(x)=\frac{1}{\lvert X\rvert}\sum_{x\in X}f(x).\]
Unless denoted otherwise, all expectations are taken over $G$. By a convenient abuse of notation, we use the same letter for the set and its characteristic function. When considering $C\subset G$, we use $\mu_C$ to denote the measure induced by $C$, so that 
\[\mu_C(x)=\frac{\abs{G}}{\abs{C}}C(x)\]
and for $A\subset G$,
\[\mu_C(A)=\langle \mu_C,A\rangle=\frac{\abs{A\cap C}}{\abs{C}}.\]
For notational convenience we will also use $\beta$ to denote $\mu_B$. Given a subspace $B\subset G$ we will denote subsets of $B$ by capital Roman letters, and their density within $B$ by the corresponding lower case Greek letter, so that $\beta(A)=\alpha$, for example. We draw special attention to our novel notation for the balanced function of a subset $A\subset B$, defined by
\[\mathbf{A}(x):=(A-\alpha B)(x)=\begin{cases}
                                  1-\alpha & \text{if }x\in A,\\
-\alpha & \text{if }x\in B\backslash A\text{, and }\\
0 &\text{otherwise,}
                                 \end{cases}
\]
so that $\bbe_{x\in B}\mathbf{A}(x)=0$. For any functions $f,g:G\to\bbc$ we define
\[\langle f,g\rangle_{\mu_X}=\bbe_{x\in X}f(x)\overline{g(x)},\]
and similarly, for any $1\leq p<\infty$,
\[\norm{f}_{p(\mu_X)}=\bigg(\bbe_{x\in X} \abs{f(x)}^{p}\bigg)^{1/p}\text{ \,\,and }\|f\|_{\infty (\mu_X)}=\sup_{x\in X}\lvert f(x)\rvert.\]

Where the measure is omitted these expectations should be taken over $G$. There is an analogue of Fourier analysis in the $\bbf_q[t]$ setting. If $x=\sum_{i<N}a_i t^i\in\bbf_q(\!(1/t)\!)$ then we define
\[e(x):=\exp(2\pi i\text{Tr}(a_{-1})/p)\]
where $\text{Tr}:\bbf_q\to\bbf_p$ is the familiar trace map. We may now define a character on $G$ as the map $x\mapsto e(\xi x)$ where $\xi$ is a member of $\mathbb{T}$, the dual group of $\bbf_q[t]$ which is defined by
\[\mathbb{T}=\left\{\sum_{i<0}a_it^i : a_i\in\bbf_q\right\}.\]
The usual definitions and results of Fourier analysis apply in this case, so we define the Fourier transform of $f$ by 
\[\widehat{f}(\xi)=\bbe_{x\in G}f(x)e(\xi x),\]
and we have Parseval's identity
\[\sum_{\xi\in\widehat{G}}\widehat{f}(\xi)\overline{\widehat{g}(\xi)}=\langle f,g\rangle\]
where $\widehat{G}$ is the dual group of $G$ defined by
\[\widehat{G}=\{a_{-1}t^{-1}+\cdots+a_{-N}t^{-N} : a_i\in\bbf_q,\,\,-N\leq i\leq -1\}.\]
For a more comprehensive discussion and detailed proofs for Fourier analysis over $\bbf_q[t]$ we refer the reader to \cite{kubota}.
In most cases, when we use the convolution operator it will be with respect to the measure over some subspace $B$ (which will be clear from the context), and hence we define
\[f\ast g(x)=\bbe_{y\in B}f(y)g(x-y).\]
In some cases, $g$ will be a measure $\mu_C$ for some other set $C$, and there we define
\[f\ast \mu_C(x)=\bbe_{y\in C}f(x-y).\]
Note in particular that when convolving with a measure, the convolution itself is always global (i.e. over the entire group $G$). As usual, the Fourier transform converts convolution into multiplication, with an appropriate scaling factor:
\[\mu_G(B)\widehat{f\ast g}(\xi)=\widehat{f}(\xi)\widehat{g}(\xi).\]
This has the useful corollary, frequently used without mention in what follows, that
\[\mu_G(B)\bbe f\ast g=\bbe f\bbe g.\]
For finite $\Gamma\subset\mathbb{T}$ and $\kappa: \Gamma\to\bbn$ we define the Bohr set $B_\kappa(\Gamma)$ to be the set
\[\{x\in G : \deg\{ x\xi \}<-\kappa(\xi)\text{ for all }\xi\in\Gamma\}\]
where if $x=\sum_{i<N}a_i t^i\in\bbf_q(\!(1/t)\!)$ then $\{ x\}=\sum_{i<0}a_i t^i$. We call the size of $\Gamma$ the rank of $B$, denoted by $\text{rk}(B)$, and call $k=\max_{\xi\in\Gamma}\kappa(\xi)$ the width of $B$. We note the easy estimate
\[\lvert B_\kappa(\Gamma)\rvert\geq q^{N-\sum_{\xi\in\Gamma}\kappa(\xi)}\geq q^{N-k\lvert \Gamma\rvert}.\]
Furthermore, note that $B_\kappa(\Gamma)$ is itself a vector space over $\bbf_q$. If $B=B_\kappa(\Gamma)$ is a given Bohr set then by $B_k$ we denote the Bohr set $B_{\kappa+k}(\Gamma)$.

 This is analogous to the traditional Bohr set in the context of $\bbz/N\bbz$,
\[\text{Bohr}_\rho(\Gamma)=\left\{ x: \norm{\frac{x\xi}{N}}<\rho\text{ for all }\xi\in\Gamma\right\}\]
where $\|\cdot\|$ measures the distance to the nearest integer and $\Gamma$ is a subset of $\bbz/N\bbz$. For both concepts, the idea is to find a set on which a given set of characters is `approximately' trivial. This is made easier in our setting of $\bbf_q[t]$, since a character can only take finitely many values even as $N\to\infty$, unlike the $\bbz/N\bbz$ case. The significant observation here is that the other important feature of Bohr sets is that they are approximately closed under dilation by the coefficients of our linear forms; in the $\bbf_q[t]$ case this follows from the inclusion $\lambda\cdot B_{\kappa+m}(\Gamma)\subset B_\kappa(\Gamma)$ which holds whenever $\deg \lambda\leq m$. We also desire that they be, at least approximately, closed under addition. This has obscured the dilation preservation feature in the traditional settings of $\bbz/N\bbz$ and $\bbf_p^N$, since this follows from closure under addition using crude estimates such as $2\cdot A\subset A+A$.

In the density increment procedure used to prove Theorem \ref{mainthm2} we will need to dilate a Bohr set by some $c\in\bbf_q[t]$ and ensure that this dilate is itself a Bohr set, so that the next iteration can be performed. The following lemma shows that this is true, and that the rank of this new Bohr set can only increase by some bounded amount. Recall that $G_m$ denotes the set of polynomials in $\bbf_q[t]$ with degree strictly less than $m$. 
\begin{lemma}\label{dilatebohr}
 If $c\in \bbf_q[t]\backslash\{0\}$ and $B=B_\kappa(\Gamma)$ is a Bohr set with width $k$ such that $B\subset G_{N-\deg c}$ then $c\cdot B$ is a Bohr set with rank at most $\text{rk}(B)+q^{\deg c}$ and width $k$.
\end{lemma}
\begin{proof}
Choose $c^{-1}\in\bbf_q(t)$ such that $c^{-1}c=1$ and let $\Gamma'=c^{-1}(\Gamma\cup G_{\deg c})$. Let $\tilde{\Gamma}=\{\{ x\} : x\in\Gamma'\}$. We have $\tilde{\Gamma}\subset\widehat{G}$ and $\lvert \tilde{\Gamma}\rvert\leq\lvert \Gamma'\rvert\leq \lvert \Gamma\rvert+q^{\deg c}$, so to complete the proof it remains to show that $c\cdot B=B_{\kappa'}(\tilde{\Gamma})$, where $\kappa'$ defined by $\kappa'(\xi)=\kappa(\gamma)$ if $\xi= c^{-1}\gamma$ for some $\gamma\in\Gamma$ and $\kappa'(\xi)=1$ otherwise.

By the orthogonality relationship of the exponential function (see Lemma 1 in \cite{kubota}) we have the identity
\[\sum_{l\in G_{\deg c}}e(c^{-1}lx)=\begin{cases}0 & \text{ when }c\nmid x\text{, and }\\q^{\deg c}&\text{ when }c\mid x.\end{cases}\]
By the definition of $e(\cdot )$, however, we also have $\sum_{l\in G_{\deg c}}e(c^{-1}lx)=q^{\deg c}$ if and only if $\text{Tr}(c^{-1}l x)=0$ for all $l\in G_{\deg c}$, which is true if and only if $\deg \{ c^{-1}lx\}<-1$ for all $l\in G_{\deg c}$ since $\bbf_q\cdot G_{\deg c}=G_{\deg c}$. It follows that $c\mid x$ if and only if $\deg \{ c^{-1}lx\}<-1$ for all $l\in G_{\deg c}$. If we have some $y\in B_{k}(\tilde{\Gamma})$, therefore, then we must have $y=cx$ for some $x\in G_{N-\deg c}$. 

Furthermore, for all $\gamma\in \Gamma$ we have $\deg\{ x\gamma\}=\deg \{yc^{-1}\gamma\}<-k$ and hence $B_{k}(\tilde{\Gamma})\subset c\cdot B$. Finally, if $x\in B$ then for any $\gamma\in \Gamma$ we have $\deg\{ cxc^{-1}\gamma\}=\deg\{ x\gamma\}<-k$ and for any $l\in G_{\deg c}$ we have $\{cx\{c^{-1}l\}\}=\{cxc^{-1}l\}=0$ so certainly $\deg\{ cxc^{-1}l\}<-1$, so that $cx\in B_{k}(\tilde{\Gamma})$ and hence $B_{k}(\tilde{\Gamma})=c\cdot B$ as required.
\end{proof}

We will frequently use both the big-$O$ notation and the Vinogradov $\ll$ symbol, where the implicit constant depends on at most $s$, $\ell$ and $q$. We will also sometimes use a $c$ or $C$ to denote some positive constant, again depending on at most $s$, $\ell$ and $q$, which may vary from line to line. All these constants may be made explicit, though we do not do so here.

We shall now define the concepts of spectra and symmetry sets. We define the $\eta$-spectrum at $G$ of a function $f$ to be
\[\Delta_\eta(f)=\{\xi\in\widehat{G} : \lvert\widehat{f}(\xi)\rvert\geq\eta\| f\|_1\},\]
As above, we suppose $B$ is some subspace of $G$ which shall be clear from context, and for any $L,K\subset B$ we define the $\eta$-symmetry set to be 
\[\text{Sym}_{\eta}(L,K)=\{x\in B : L\ast K(x)\geq\eta\}.\]
Note that $\text{Sym}_{\eta}(L,-K)=-\text{Sym}_\eta(-L,K)$.
We will need some structural information on spectra due to Chang \cite{chang}. The following version is proved as Lemma 4.36 in \cite{taovu}.
\begin{lemma}\label{chang}
 Let $G$ be any finite additive group and $D\subset G$ with density $\delta$. If $\eta>0$ then there is a subset $\tilde{\Delta}\subset\Delta_\eta(D)$ such that $\lvert \tilde{\Delta}\rvert\ll\eta^{-2}\log(1/\delta)$ and
\[\Delta_\eta(D)\subset\{-1,0,1\}^{\lvert \tilde{\Delta}\rvert}\cdot\tilde{\Delta}.\]
\end{lemma}

We now use this to convert the fact that we have large Fourier coefficients over a large spectrum into a more useful density increment property. The idea goes back to the work of Heath-Brown \cite{heath-brown} and Szemer\'{e}di \cite{szemeredi}, and was developed into the following form by Sanders \cite{sanders2}.
\begin{lemma}\label{hbs}
 For any Bohr set $B\subset G$ and $\eta>0$ if $A,D\subset B$ satisfy
\[\sum_{\gamma\in\Delta_\eta(D)}\lvert\widehat{\mathbf{A}}(\gamma)\rvert^2\geq\nu\alpha^2\mu_G(B)\]
then there is a Bohr set $B'\subset B$ with
\[\text{rk}(B')\leq \text{rk}(B)+O(\eta^{-2}\log(1/\delta))\text{ and }\beta(B')\geq\exp(-C\eta^{-2}\log(1/\delta))\]
such that $\| A\ast\beta'\|_\infty\geq\alpha(1+\nu)$.
\end{lemma}
\begin{proof}
Let $\widehat{B}\subset\widehat{G}$ be the dual group of $B$ viewed as an additive group. It is easy to check that for any $\gamma\in\widehat{G}$ there is some $\gamma'\in\widehat{B}$ such that $\gamma(x)=\gamma'(x)$ for all $x\in B$. Let $\Delta'=\widehat{B}\cap\Delta_\eta(D)$. Renormalising the Fourier transform, we see that $\Delta'$ is the $\eta$-spectrum at $B$ for $D$. Let $\tilde{\Delta}\subset\Delta'$ be the set given by Lemma \ref{chang}, and refine the Bohr set $B$ by setting
\[ B'=\{ x\in B : \deg\{ x\gamma \}<-1\text{ for all }\gamma\in\tilde{\Delta}\}.\]
It follows that 
\[\text{rk}(B')\leq \text{rk}(B)+\lvert \tilde{\Delta}\rvert\leq\text{rk}(B)+O(\eta^{-2}\log(1/\delta)),\]
and hence $\beta(B')\geq\exp(-C\eta^{-2}\log(1/\delta))$. Furthermore, for all $\gamma\in\Delta_\eta(D)$ we have $\widehat{\beta'}(\gamma)=\bbe_{x\in B'}e(x\gamma)=1$.
By Parseval's identity and the initial hypothesis
\[ \|\mathbf{A}\ast\beta'\|_2^2=\sum_{\gamma\in\widehat{G}}\lvert \widehat{\mathbf{A}}(\gamma)\rvert^2\lvert\widehat{\beta'}(\gamma)\rvert^2\geq\sum_{\gamma\in\Delta_\eta(D)}\lvert\widehat{\mathbf{A}}(\gamma)\rvert^2\geq\nu\alpha^2\mu_G(B).\]
Furthermore, note that $B\ast\beta'(y)=\bbe_{x\in B'}B(y-x)=B(y)$, and hence $\langle B\ast \beta',B\ast\beta'\rangle=\mu_G(B)$.
Similarly,
\[\langle A\ast\beta',B\ast\beta'\rangle=\mu_G(B)\bbe_{x\in B}\bbe_{y\in B'}A(y-x)=\alpha\mu_G(B).\]
In particular,
\begin{eqnarray*}
\|\mathbf{A}\ast \beta'\|_2^2&=&\| A\ast\beta'\|_2^2+\alpha^2\langle B\ast\beta',B\ast\beta'\rangle-2\alpha\langle A\ast\beta',B\ast\beta'\rangle\\
&=&\| A\ast\beta'\|_2^2-\alpha^2\mu_G(B).
\end{eqnarray*}
It follows that $\| A\ast \beta'\|_2^2\geq\alpha^2(1+\nu)\mu_G(B)$. Combining H\"{o}lder's inequality with the equality $\| A\ast\beta'\|_1=\alpha\mu_G(B)$ gives $\| A\ast\beta'\|_\infty\geq\alpha(1+\nu)$ as required.
\end{proof}
We make one final remark on the values of the densities $\alpha$ of the various $A\subset G$ we consider, where these $A$ will be sets with no non-trivial solutions to the equation \eqref{main}. Recall that we may assume $\alpha\leq r(N)\ll1/\log N$ by adapting the method of Roth. In particular, for sufficiently large $N$, we will assume that $\alpha<1/e^2$, so that $\log(1/\alpha)>2$. Making this assumption avoids some notational awkwardness.
\section{A combinatorial transformation}\label{kk}
We begin with a generalisation of one of the two main tools in the method of Sanders, a combinatorial transformation. This has its origins in the work of Katz and Koester \cite{katzkoester}, though a similar transformation known as the Dyson $e$-transform has been a useful tool in additive combinatorics since \cite{dyson}.

The idea is to find, given subsets $A_1,\hdots,A_{k+1}$ of a Bohr set $B$, corresponding $L,S_1,\hdots, S_k\subset B$ such that $L$ is dense inside $B$, the $S_i$ are not too sparse and 
\[L\ast S_1\ast\cdots\ast S_k\leq\alpha_1^{-2}A_1\ast A_2\ast\cdots\ast A_{k+1}.\]
 Since solutions to \eqref{main} are counted by $\langle (c_1\cdot A)\ast\cdots\ast(c_{s-1}\cdot A),c_s\cdot A\rangle$ to find many solutions to \eqref{main} in $A$ it then suffices to find a lower bound for $\langle L\ast \cdots\ast S_{s-2},c_s\cdot A\rangle$. The large density of $L$ may then be efficiently exploited using the recent probabilistic method of Croot and Sisask, as we shall do in the next section. When we cannot find such sets we find a strong density increment and jump to the final stage of the proof.

We first give a simplified proof of Lemma 4.2 from \cite{sanders1}, taking advantage of the fact that our Bohr sets are subspaces in order to streamline the proof. We also observe that the proof gives a slightly stronger statement than is recorded in \cite{sanders1}. We then apply this technical lemma iteratively to construct sets $L,S_1,\hdots, S_k$ as above to prove the main theorem of this section. This is a generalisation of Proposition 4.1 from \cite{sanders1}, taking advantage of the large number of convolutions to run a more efficient iterative procedure.

\begin{lemma}\label{kklemma}
  If $B$ is a Bohr set and $K,T,L,S\subset B$ then either
\begin{enumerate}
 \item there exists a Bohr set $B'\subset B$ with rank at most $\text{rk}(B)+O(\lambda\kappa^{-1}\log(1/\tau))$, density $\beta(B')\geq\exp(-C\lambda\kappa^{-1}\log(1/\tau))$ such that, for some $x\in B$, $\beta'(K+x)\geq\kappa/32\lambda$, or
\item there are $L',S'\subset B$ with $\beta(L')\geq\lambda+\kappa/2$ and $\beta(S')\geq\tau\sigma/2$ such that, for all $x\in G$,
\[L'\ast S'(x)\leq L\ast S(x)+K\ast T(x).\]
\end{enumerate}
 \end{lemma}
\begin{proof}
 For an arbitrary $x\in B$, let $L_x=L\cup(K+x)$ and $S_x=S\cap(T-x)$. By construction,
\[ L_x\ast S_x\leq L\ast S_x+(K+x)\ast S_x\leq L\ast S+K\ast T.\]
Furthermore, we have
\[\beta(L_x)=\lambda+\kappa-\beta(L\cap (K+x))=\lambda+\kappa-L\ast (-K)(x)\]
and $\beta(S_x)=(-S)\ast T(x)$.
Hence if $\lvert\text{Sym}_{\tau\sigma/2}(-S,T)\rvert>\lvert\text{Sym}_{\kappa/2}(L,-K)\rvert$ then we are in the second case by taking $L'=L_{x}$ and $S'=S_x$ for some $x\in \text{Sym}_{\tau\sigma/2}(-S,T)\backslash\text{Sym}_{\kappa/2}(L,-K)$. Note that
\[\beta(\text{Sym}_{\tau\sigma/2}(-S,T))\sigma+\sigma\tau/2\geq\bbe_{x\in B}(-S)\ast T(x)=\sigma\tau,\]
so that $\beta(\text{Sym}_{\tau\sigma/2}(-S,T))\geq\tau/2$. Hence either we are in the second case of the theorem, or we may assume that $D=\text{Sym}_{\kappa/2}(-L,K)$ has density at least $\tau/4$, say. We have
\[\langle (-L)\ast K,D\rangle_{\beta}\geq\kappa\delta/2.\]
Note that
\[\langle (-L)\ast B,D\rangle_{\beta}=\lambda\delta.\]
We may assume that $\lambda<1/4$, say, or else the conclusion is trivial. Hence, by the triangle inequality,
\[\lvert\langle(-L)\ast\mathbf{K},D\rangle_{\beta}\rvert\geq\kappa\delta(1/2-\lambda)\geq\kappa\delta/4.\]
Taking the Fourier transform of the left hand side we see that 
\[\left\lvert\sum_{\gamma\in\widehat{G}}\widehat{(-L)}(\gamma)\widehat{\mathbf{K}}(\gamma)\widehat{D}(\gamma)\right\rvert\geq\kappa\delta\mu_G(B)^2/4.\]
The Cauchy-Schwarz inequality combined with Parseval's identity then gives
\[\sum_{\gamma\in\widehat{G}}\lvert \widehat{\mathbf{K}}(\gamma)\rvert^2\lvert\widehat{D}(\gamma)\rvert^2\geq\kappa^2\delta^2\mu_G(B)^3/16\lambda.\]
Furthermore, if we let $\eta^2=\kappa/32\lambda$ then by Parseval's identity again we have
\[\sum_{\gamma\notin\Delta_\eta(D)}\lvert \widehat{\mathbf{K}}(\gamma)\rvert^2\lvert\widehat{D}(\gamma)\rvert^2\leq\kappa^2\delta^2\mu_G(B)^3/32\lambda.\]
Using the trivial upper bound $\lvert\widehat{D}(\gamma)\rvert\leq\delta\mu_G(B)$ it follows that
\[\sum_{\gamma \in \Delta_\eta(D)}\lvert\widehat{\mathbf{K}}(\gamma)\rvert^2\geq\kappa^2\mu_G(B)/32\lambda.\]

Applying Lemma \ref{hbs}, we have a Bohr set $B'\subset B$ with rank and size as required and $\beta'(K+x)\geq\kappa/32\lambda$ for some $x\in X$. 
\end{proof}
We now apply Lemma \ref{kklemma} iteratively to prove the main theorem of this section. 
\begin{theorem}\label{kkthm}
Let $B$ be a Bohr set and suppose $A_1\subset B$ with density $\alpha_1$ and $A_2,\hdots,A_{k+1}\subset B$ all with density at least $\alpha$. Then either
\begin{enumerate}
 \item  there exists a Bohr set $B'\subset B$ with rank $\mbox{rk}(B')\leq \mbox{rk}(B)+O(\alpha_1^{-1/k}\log(1/\alpha))$, density $\beta(B')\geq\exp(-C\alpha_1^{-1/k}\log(1/\alpha))$ and, for some $x\in B$, we have $\beta'(A_1+x)\geq2\alpha_1$, or
 \item there are sets $L\subset B$ and $S_1,\hdots,S_k\subset B$ with $\lambda\geq2^{-k-6}$ and $\sigma_i\geq\alpha^{C\alpha_1^{-1/k}}$ for $1\leq i\leq k$ such that
\[L\ast S_1\ast\cdots\ast S_k(x)\leq\alpha_1^{-2}A_1\ast A_2\ast\cdots\ast A_{k+1}(x)\]
 for all $x\in G$. 
\end{enumerate}
\end{theorem}
\begin{proof}
Let $L_0=A_1$. We show by induction on $1\leq j\leq k$ that we may either find $X_1,\hdots, X_j\subset B$ such that $L_j=A_1+X_1+\cdots+X_j$ satisfies 
\[\lambda_j\geq\max(\alpha_1\lvert X_1\rvert\cdots\lvert X_j\rvert/2^j,\alpha_1^{1-j/k}/2^j)\]
 and $S_1,\hdots,S_j\subset B$  satisfying $\sigma_i\geq\alpha^{C\alpha_1^{-1/k}}$ for $1\leq i\leq j$ such that
\[L_j\ast S_1\ast\cdots\ast S_j \ast A_{j+2}\ast\cdots\ast A_{k+1}(x)\leq\alpha_1^{-2j/k}A_1\ast A_2 \ast\cdots\ast A_{k+1}(x)\]
 for all $x\in G$, or there is a Bohr set $B'\subset B$ such that $\mbox{rk}(B')\leq \mbox{rk}(B)+O(\alpha_1^{-1/k}\log(1/\alpha))$ and for some $x\in B$ we have  $\beta'(L_{j-1}+x)\geq2\alpha_1\lvert X_1\rvert\cdots\lvert X_{j-1}\rvert$. If we are in the second case for any $1\leq j\leq k$ then we halt the inductive procedure and by the pigeonhole principle we must have, for some $x_i\in X_i$, $\beta'(A_1+x_1+\cdots+x_{j-1}+x)\geq2\alpha_1$, giving us the first case of the theorem. Otherwise, setting $j=k$ and $L=L_k$ gives the second case of the theorem.

We now fix $1\leq j\leq k$ and suppose that we have produced $L_{j-1}$ and $S_1,\hdots,S_{j-1}$ as above. We produce a sequence of sets $L_j^{(i)}, S^{(i)}_j\subset B$ iteratively such that $\lambda_{j-1}i\geq \lambda_j^{(i)}\geq\lambda_{j-1} i/2$, $\sigma^{(i)}_j\geq(\alpha/2)^{i}$, and $L_j^{(i)}\ast S_j^{(i)}\leq iL_{j-1}\ast A_{j+1}$.
We begin the iteration by letting $L_j^{(1)}:=L_{j-1}$ and $S_j^{(1)}:=A_{j+1}$. 

We now repeatedly apply Lemma \ref{kklemma}, with $K=L_{j-1}$ and $T=A_{j+1}$, and note that an examination of the proof shows that we take $L_j^{(i+1)}=L_j^{(i)}\cup(L_{j-1}+x_{i+1})$ for some $x_{i+1}\in B$ so that $L_j^{(i)}=L_{j-1}+X_j^{(i)}$ where $\lvert X_j^{(i)}\rvert\leq i$. If we can continue the process for $i'=\lceil\alpha_1^{-1/k}\rceil$ steps, then we can halt and set $L_j=L_j^{(i')}$, so that $L_j=L_{j-1}+X_j$ where \[\lambda_j\geq\lambda_{j-1}\lceil\alpha_1^{-1/k}\rceil/2\geq\max(\alpha_1\lvert X_1\rvert\cdots\lvert X_j\rvert/2^j,\alpha_1^{1-j/k}/2^j)\]
by inductive hypothesis. Furthermore, if we set $S_j=S_j^{(i')}$ then $\sigma_j\geq(\alpha/2)^{i'}\geq \alpha^{C\alpha_1^{-1/k}}$ and by inductive hypothesis again,
\begin{eqnarray*}L_j\ast \cdots\ast S_j \ast A_{j+2}\ast\cdots\ast A_{k+1}&\leq& \lceil\alpha_1^{-1/k}\rceil L_{j-1}\ast\cdots\ast S_{j-1} \ast A_{j+1}\ast\cdots\ast A_{k+1}\\
 &\leq& \alpha_1^{-2j/k}A_1\ast A_2 \ast\cdots\ast A_{k+1}
\end{eqnarray*}
as required. 

Otherwise, we must halt at some $i'\leq \alpha_1^{-1/k}$ with the first alternative from Lemma \ref{kklemma}, so that there is a Bohr set $B'\subset B$ with rank at most \[\mbox{rk}(B)+O(\lambda_j^{(i')}\lambda_{j-1}^{-1}\log(1/\alpha))\leq\mbox{rk}(B)+O(\alpha_1^{-1/k}\log(1/\alpha))\]
and density $\beta(B')\geq \exp(-C\alpha_1^{-1/k}\log(1/\alpha))$ such that, for some $x\in B$, \[\beta'(L_{j-1}+x)\geq\lambda_{j-1}/32\lambda_j^{(i')}\geq\alpha_1\lvert X_1\rvert\cdots\lvert X_{j-1}\rvert/2^{j+5}\lambda_j^{(i')}\geq 2\alpha_1\lvert X_1\rvert\cdots\lvert X_{j-1}\rvert\]
 since if $\lambda_j^{(i')}>2^{-j-6}>2^{-k-6}$ we may end the entire procedure with the second case of the theorem. This completes the induction.
\end{proof}

\section{The Croot-Sisask probabilistic method}\label{cs}
We now use the new Croot-Sisask method of random sampling from \cite{crootsisask} to find a density increment that can exploit the fact that one set in a convolution is very dense, while allowing us to ignore the density loss in the other sets that results from applying Theorem \ref{kkthm}. The operator $\tau_t$ is defined by $\tau_t(f)(x)=f(x-t)$. The theorem finds a large set $T$ such that the translation operator $\tau_t$ is essentially constant on the convolution $L\ast S_1\ast\cdots\ast S_k$ for all $t\in T$. This set may have no structure initially, but Croot and Sisask observed that we may recover some structure by taking repeated sumsets of $T$. In \cite{sanders1} this was exploited to great effect as in the proof below.

We first state the result from \cite{crootsisask} that we require, before proving an analogue of Corollary 5.2 of \cite{sanders1} for multiple convolutions. The proof given here is very similar, but once again we are able to present a simplified and streamlined version using the subspace structure of our Bohr sets.
 
\begin{theorem}\label{cslemma}
 Let $B$ be an additive group, $S\subset B$ and $f:B\to\bbc$ any function. Then for any $\epsilon>0$ and $p\geq 2$ there is some $T\subset B$ with $\beta(T)\geq \sigma^{C\epsilon^{-2}p}$ such that
\[\| \tau_t(f\ast \mu_{S})-f\ast \mu_S\|_{p(\beta)}\leq\epsilon\| f\|_{p(\beta)}\]
for all $t\in T$.
\end{theorem}
This is a global version of Theorem 5.1 from \cite{sanders1}, which we are able to apply here because Bohr sets in $\bbf_q[t]$ are additive groups. When they are only approximately closed under addition, as in $\bbz/N\bbz$, the local theorem of \cite{sanders1} must be used. The following theorem applies this to obtain a useful density increment.
\begin{theorem}\label{csthm}
Let $B$ be a Bohr set, and suppose $A,L,S_1,\hdots,S_k\subset B$. Then either
 \begin{enumerate}
  \item $\langle L\ast S_1\ast\cdots\ast S_k,A\rangle_\beta\geq \lambda\sigma_1\cdots\sigma_k\alpha/2$, or
  \item for any integer $l$ there exists a Bohr set $B'\subset B$ such that
\[\text{rk}(B')\leq \text{rk}(B)+O(\lambda^{-2-1/l}\alpha^{-1/l}l^2\log(1/\alpha)\log(1/\sigma_k)),\]
\[\beta(B')\geq\exp(-C\lambda^{-2-1/l}\alpha^{-1/l}l^2\log(1/\alpha)\log(1/\sigma_k))\]
and, for some $x\in B$, we have $\beta'(A+x)\geq\alpha(1+c\lambda)$.
 \end{enumerate}
\end{theorem}
\begin{proof}
 Applying Theorem \ref{cslemma} with parameters $p$ and $\epsilon$ to be chosen later, where $f=L\ast \mu_{S_1}\ast\cdots\ast \mu_{S_{k-1}}$, there is a set $T\subset B$ such that $\lvert T\rvert\geq\sigma_k^{C\epsilon^{-2}p}\lvert B\rvert$ and
\begin{eqnarray*}
\| \tau_t(L\ast \mu_{S_1}\ast\cdots\ast \mu_{S_k})-L\ast\mu_{S_1}\ast\cdots\ast\mu_{S_k}\|_{p(\beta)}&\leq&\epsilon\|L\ast \mu_{S_1}\ast\cdots\ast \mu_{S_{k-1}}\|_{p(\beta)}\\&\leq& \epsilon\| L\|_{p(\beta)}
\end{eqnarray*}

 for all $t\in T$, or after rescaling,
\[ \|\tau_t(L\ast S_1\ast\cdots\ast S_k)-L\ast S_1\ast\cdots\ast S_k\|_{p(\beta)}\leq\epsilon\sigma_1\cdots\sigma_k\lambda^{1/p}\leq\epsilon\sigma_1\cdots\sigma_k.\]
By the triangle inequality, for any $l\in\bbn$ and $t\in lT$,
\[\|\tau_t(L\ast S_1\ast\cdots\ast S_k)-L\ast S_1\ast\cdots\ast S_k\|_{p(\beta)}\leq l\epsilon\sigma_1\cdots\sigma_k.\]
If we let $g:=\mu_{T}\ast\cdots\ast\mu_{T}$ with $l$ copies of $\mu_{T}$, then after recalling the definition of $\tau_t$ and applying the triangle inequality again we get
\[\| L\ast S_1\ast\cdots\ast S_k\ast g-L\ast S_1\ast\cdots\ast S_k\|_{p(\beta)}\leq l\epsilon\sigma_1\cdots\sigma_k.\]
By H\"{o}lder's inequality,
\begin{eqnarray*}
 \left\lvert \langle L\ast S_1\ast\cdots\ast S_k\ast g,A\rangle_{\beta}-\langle L\ast S_1\ast\cdots\ast S_k,A\rangle_{\beta}\right\rvert&\leq&l\epsilon\sigma_1\cdots\sigma_k\| A\|_{p/p-1(\beta)}\\&\leq&\lambda\sigma_1\cdots\sigma_k\alpha/4
\end{eqnarray*}

if we let $p=\lceil\log(1/\alpha)\rceil$ and $\epsilon=\lambda/4el$.
Hence either we are in the first case, or
\[\langle L\ast S_1\ast\cdots\ast S_k\ast g,A\rangle_{\beta}\leq3\lambda\sigma_1\cdots\sigma_k\alpha/4.\]
Since $\bbe_{x\in B}g=1$ we have
\[\bbe_{x\in B} L\ast S_1\ast\cdots\ast S_k\ast g(x)=\lambda\sigma_1\cdots\sigma_k.\]
Hence
\[\lvert\langle L\ast S_1\ast\cdots\ast S_k\ast g,\mathbf{A}\rangle_{\beta}\rvert\geq\lambda\sigma_1\cdots\sigma_k\alpha/4.\]
We complete the proof by a standard conversion into Fourier space, as in the proof of Lemma \ref{kklemma}. Taking the Fourier transform of this inequality we see that
\[\left\lvert \sum_{\gamma\in\widehat{G}}\widehat{L}(\gamma)\widehat{S_1}(\gamma)\cdots\widehat{S_k}(\gamma)\widehat{g}(\gamma)\widehat{\mathbf{A}}(\gamma)\right\rvert\geq \lambda\sigma_1\cdots\sigma_k\alpha\mu_G(B)^{k+1}/4.\]
Using the Cauchy-Schwarz inequality combined with Parseval's identity and the trivial bound $\lvert\widehat{S_i}(\gamma)\rvert\leq \sigma_i\mu_G(B)$ gives
\[\sum_{\gamma\in\widehat{G}}\left\lvert\widehat{g}(\gamma)\widehat{\mathbf{A}}(\gamma)\right\rvert^2\geq \lambda\alpha^2\mu_G(B)/16.\]
Recalling the definition of $g$ this is
\[\sum_{\gamma\in\widehat{G}}\left\lvert\widehat{\mu_T}(\gamma)\right\rvert^{2l}\left\lvert\widehat{\mathbf{A}}(\gamma)\right\rvert^2\geq \lambda\alpha^2\mu_G(B)/16.\]
Furthermore, if we let $\eta^{2l}=\lambda\alpha/32$ then, by Parseval's identity,
\[\sum_{\gamma\not\in\Delta_\eta(T)}\left\lvert\widehat{\mu_T}(\gamma)\right\rvert^{2l}\left\lvert\widehat{\mathbf{A}}(\gamma)\right\rvert^2\leq \lambda\alpha^2\mu_G(B)/32.\]
Combining these inequalities we have
\[\sum_{\gamma\in\Delta_\eta(T)}\left\lvert\widehat{\mathbf{A}}(\gamma)\right\rvert^2\geq \sum_{\gamma\in\Delta_\eta(T)}\left\lvert\widehat{\mu_T}(\gamma)\right\rvert^{2l}\left\lvert\widehat{\mathbf{A}}(\gamma)\right\rvert^2\geq \lambda\alpha^2\mu_G(B)/32\]
and the second case follows from Lemma \ref{hbs}.
\end{proof}
\section{Density increment strategy}\label{disection}
We now prove Theorem \ref{mainthm2} using the traditional density increment strategy: we construct a series of Bohr sets $B^{(i)}$ such that at each stage, either $A\cap B^{(i)}$ has many solutions to \eqref{main}, or we may find a new Bohr set $B^{(i+1)}\subset B^{(i)}$ on which the density of $A$ increases. Since the density is always trivially bounded above by $1$, we can only iterate this a bounded number of times before we must end with the first alternative, from which we deduce the existence of many solutions to \eqref{main} in the original set $A$, and hence deduce that $A$ must contain at least one non-trivial solution if $\alpha$ is sufficiently large.

We define  $\Lambda(A)$ by letting $\Lambda(A)\lvert G\rvert^{s-1}$ be the number of solutions $(x_1,\hdots,x_s)$ to \eqref{main} with $x_i\in A$. Note in particular that we include the trivial solutions so that $\Lambda(A)\geq \lvert A\rvert^m\lvert G\rvert^{1-s}$, where $m$ is the genus of \eqref{main}. 

We first combine the results from the previous two sections to obtain a density increment which we then iterate to provide a lower bound for $\Lambda(A)$, from which we deduce Theorem \ref{mainthm2} as a corollary.

\begin{lemma}\label{di}
Let $B$ be a Bohr set and $A_i\subset B$ for $1\leq i\leq s$, all with density at least $\alpha$. Then either
\begin{enumerate}
 \item $\langle A_1\ast A_2\cdots\ast A_{s-1},A_s\rangle_\beta\gg\alpha^{C\alpha^{-1/(s-2)}}$, or
\item we have a Bohr set $B'\subset B$ of rank $\text{rk}(B)+O(\alpha^{-1/(s-2)}(\log(1/\alpha))^4)$ and size 
\[\beta(B')\geq\exp(-C\alpha^{-1/(s-2)}(\log(1/\alpha))^4)\]
such that for some $x\in B$ and $i\in\{1,s\}$ we have $\beta'(A_i+x)\geq\alpha(1+c)$.
\end{enumerate}
\end{lemma}
\begin{proof}
 By Theorem \ref{kkthm}, either we are in the second case of the theorem or there are $L,S_1,\hdots,S_{s-2}\subset B$ with $\lambda>2^{-s-3}$ and $\sigma_i\geq\alpha^{C\alpha^{-1/(s-2)}}$ for $1\leq i\leq s-2$ such that
\[A_1\ast A_2\cdots\ast A_{s-1}\geq\alpha^2 L\ast S_1\ast\cdots\ast S_{s-2}.\]

We may now apply Theorem \ref{csthm} to get that either for any integer $l$ we have a Bohr set $B'\subset B$ such that
\[\mbox{rk}(B')\leq \mbox{rk}(B)+O(\alpha^{-1/(s-2)-1/l}l^2\log^2(1/\alpha))\]
and, for some $x\in B$, we have $\beta'(A_s+x)\geq\alpha(1+c)$, or 
\[\langle L\ast S_1\ast\cdots\ast S_{s-2},A_s\rangle_\beta\geq\sigma_1\cdots\sigma_{s-2}\lambda\alpha/2\gg\alpha^{C\alpha^{-1/(s-2)}}.\]
Combining this with the bound above, we are either in the first or second case of the theorem, after taking a roughly optimal choice of $l=\lceil\log(1/\alpha)\rceil$.
\end{proof}

We now prove the following generalisation of Theorem \ref{mainthm2} using the traditional density increment strategy.
\begin{theorem}\label{big}
Let $s\geq 3$, and suppose that $\mathbf{c}\in\bbf_q[t]^s$ satisfies the condition that \eqref{main} has genus $m\geq 1$. Then there exists a constant $C>0$, depending only on $s$, $q$ and $\ell=\max(\deg c_i)$, such that if $A\subset G$ then
\[\Lambda(A)\geq\exp(-C\alpha^{-1/(s-2)}(\log(1/\alpha))^4)	.\]
\end{theorem}
\begin{proof}
We shall iteratively construct a sequence of Bohr sets $B^{(i)}$ and sets $A^{(i)}\subset G$, each a subset of some translated and dilated copy of $A$. Let $\text{rk}(B^{(i)})=d^{(i)}$ and $\beta^{(i)}(A^{(i)})=\alpha^{(i)}$. We shall insist that, given $B^{(i)}$ and $A^{(i)}$, we can either find a $B^{(i+1)}$ and $A^{(i+1)}$ such that for some constant $c'>0$ (depending only on $s$, $\ell$ and $q$), the density satisfies $\alpha^{(i+1)}\geq\alpha^{(i)}(1+c')$,
\[d^{(i+1)}\leq d^{(i)}+O(\alpha^{(i)(-1/(s-2))}(\log\alpha^{-(i)})^4),\]
and
\[\mu_{B^{(i)}}(B^{(i+1)})\geq\exp(-C(d^{(i)}+\alpha^{(i)(-1/(s-2))}(\log\alpha^{-(i)})^4)),\]
or we have
\[\Lambda(A)\geq \exp(-Cd^{(i)})\mu_{G}(B^{(i)})^{s-1}\alpha^{C\alpha^{-1/(s-2)}}.\]

We begin the iteration by letting $B^{(1)}=G_{N-s\ell}$ so that $d^{(1)}\leq q^{s\ell}$ and $A^{(1)}=A'\cap B^{(1)}$ where $A'$ is some translate of $A$ chosen by the pigeonhole principle so that $\alpha^{(1)}\geq \alpha$. This is permissible since the conclusion of our theorem is invariant under translation. Since we always have the trivial bound $\alpha^{(i)}\leq 1$, we can iterate this process only $O(\log(1/\alpha))$ times, hence at some $B^{(K)}$, with $K\ll\log(1/\alpha)$, we have the second alternative. Here we have
\begin{eqnarray*}
 d^{(K)}\ll \sum_{i=1}^K\alpha^{(i)(-1/(s-2))}(\log\alpha^{-(i)})^4&\ll&\alpha^{-1/(s-2)}(\log(1/\alpha))^4\sum_{i=0}^K(1+c')^{-i}\\&\ll&\alpha^{-1/(s-2)}(\log(1/\alpha))^4.
\end{eqnarray*}

Similarly,
\[\mu_G(B^{(K)})\geq\exp(-C\alpha^{-1/(s-2)}(\log(1/\alpha))^4)\]
from which the result follows.

 We now discuss the iterative procedure. We are given a Bohr set $B^{(i)}$ and a set $A^{(i)}$. To reduce notation in what follows, let $\widetilde{A}=A^{(i)}\cap B^{(i)}$ and $B=B^{(i)}$. We have $\widetilde{A}\subset B$ with density $\widetilde{\alpha}=\alpha^{(i)}$, where $B$ is a Bohr set with rank $d$ and width $k$. Note that for any $B'\subset B$ we have $\bbe_{x\in B}\widetilde{A}\ast\beta'(x)=\bbe_{x\in B'}\widetilde{A}\ast\beta(x)=\widetilde{\alpha}$. In particular, for $1\leq j\leq s$ define $B_j=c_1\cdots c_{j-1}c_{j+1}\cdots c_s\cdot B_{s\ell}$. Since $B_j\subset B$ for all $1\leq j\leq s$, we have
\[\bbe_{x\in B}(\widetilde{A}\ast\beta_1+\cdots+\widetilde{A}\ast\beta_s)(x)=s\widetilde{\alpha}\]
and in particular there is some $x\in B$ such that
\[\widetilde{A}\ast\beta_1(x)+\cdots+\widetilde{A}\ast\beta_s(x)\geq s\widetilde{\alpha}.\]

Note that if 
\[\widetilde{A}\ast\beta_j(x)\geq\widetilde{\alpha}\left(1+\frac{c}{2(s-1)}\right)\]
for any $1\leq j\leq s$ then we may proceed to the next stage of the iteration, letting $B^{(i+1)}=B_j$ and $A^{(i+1)}=A^{(i)}-x$. Here we use Lemma \ref{dilatebohr} which ensures $B_j$ is a genuine Bohr set, after noting that since we chose $B^{(1)}=G_{N-s\ell}$ the condition $B_{s\ell}\subset G_{N-\deg\{c_1\cdots c_{j-1}c_{j+1}\cdots c_s\}}$ will always be satisfied. Otherwise we must have $\widetilde{A}\ast\beta_j(x)\geq\widetilde{\alpha}\left(1-\frac{c}{2}\right)$ for all $1\leq j\leq s$.

Hence if we let $B'=c_1\cdot B_1=\cdots=c_s\cdot B_s$, and similarly $A_j=c_j(A^{(i)}-x)$, then we have $\beta'(A_j)\geq\widetilde{\alpha}\left(1-\frac{c}{2}\right)$ for $1\leq j\leq s$. By Lemma \ref{dilatebohr} we see that $B'$ is a genuine Bohr set and furthermore $B'\subset B$. We now satisfy the conditions of Lemma \ref{di}; hence either 
\[\langle A_1\ast A_2\cdots\ast A_{s-1},A_s\rangle_{\beta'}\geq \alpha^{C\alpha^{-1/(s-2)}},\]
or there is a Bohr set $B''\subset B'$ of rank $d+O(\alpha^{-1/(s-2)}(\log(1/\alpha))^4)$ and density
\[\beta(B'')=\beta(B')\beta'(B'')\geq\exp(-C(d+\alpha^{-1/(s-2)}(\log(1/\alpha))^4))\]
such that either $A_1$ or $A_s$ has density increment at least $\widetilde{\alpha}(1+c)(1-\frac{c}{2})\geq\widetilde{\alpha}(1+\frac{c}{4})$. In this case we are done, letting $B^{(i+1)}=B''$ and $A^{(i+1)}$ be $A_1$ or $A_s$ as appropriate.

Otherwise, note that
\[\Lambda(A)\geq \mu_G(B')^{s-1}\langle A_1\ast A_2\cdots\ast A_{s-1},A_s\rangle_{\beta'}.\]
Here is where we make essential use of the translation and dilation invariance condition $\sum c_i=0$, for the inner product on the right hand side only counts solutions in a set which may be translated and dilated many times from our original $A$, but each solution can be retranslated and dilated to give a solution in the original set.
Putting this all together and noting the estimate $\beta(B')\geq\exp(-C d)$ we have
\[\Lambda(A)\geq \mu_G(B')^{s-1}\alpha^{C\alpha^{-1/(s-2)}}\geq \exp(-Cd)\mu_G(B)^{s-1}\alpha^{C\alpha^{-1/(s-2)}}\]
as required.
\end{proof}

We finally prove Theorem \ref{mainthm2} as a corollary.
\begin{proof}[of Theorem~{\rm\ref{mainthm2}}]
 Suppose $A\subset G$ with density $\alpha$ has no non-trivial solutions to \eqref{main}, so that $\Lambda(A)=\alpha^m q^{(1+m-s)N}$.
Theorem \ref{big} then implies $N\ll \alpha^{-1/(s-2)}(\log(1/\alpha))^4$. If $\alpha\ll N^{-s+2}$ then we are done; otherwise, we have $\log(1/\alpha)\ll\log N$, and hence
\[\alpha\ll\left(\frac{(\log N)^4}{N}\right)^{s-2}\]
as required.
\end{proof}
\section{The integer case}\label{integercase}

We now sketch how to prove Theorem \ref{mainthm1} using the ideas from previous sections. The results here will be much closer in spirit to \cite{sanders1}, and we focus on explaining the process of translating a proof in $\bbf_q[t]$ to one in $\bbz/N\bbz$. 

Given $\rho>0$ and $\Gamma\subset \bbz/N\bbz$ we define the Bohr set $B_\rho(\Gamma)$ to be
\[B_\rho(\Gamma)=\{ x\in \bbz/N\bbz : \lvert e(\gamma x)-1\rvert<\rho\text{ for all }\gamma\in\Gamma\}\]
where we are now using the character given by $e(x)=\exp(2\pi i x/N)$. As in the $\bbf_q[t]$ case, we define $\text{rk}(B)=\lvert \Gamma\rvert$ and call $\rho$ the width of $B_\rho(\Gamma)$. We have the analogous size estimate $\lvert B_\rho(\Gamma)\rvert\geq \rho^{\lvert \Gamma\rvert}N$. Further discussion and proofs can be found in Chapter 4 of \cite{taovu}.

The reason for the extra technical complications in the $\bbz/N\bbz$ case is that it is no longer true that $B+B=B$ for a Bohr set $B$. In general, the best we can do is $B_\rho(\Gamma)+B_\rho(\Gamma)\subset B_{2\rho}(\Gamma)$, which may be exponentially larger. The key insight here, originating with Bourgain \cite{bourgain1}, is that if we restrict ourselves to instead consider $B+B'$ for some smaller Bohr set $B'\subset B$, and if $B$ is sufficiently well-behaved, then we know that $B+B'$ is {\em mostly} contained within $B$, and this is enough to run the previous sort of arguments while incurring only a small error term.

To make this precise we call a Bohr set $B=B_\rho(\Gamma)$, where $\text{rk}(B)=k$, regular if 
\[\frac{\lvert B\rvert}{1+100k\lvert \eta\rvert}\leq \lvert B_{(1+\eta)\rho}(\Gamma)\rvert\leq (1+100k\lvert \eta\rvert)\lvert B\rvert\]
whenever $\lvert \eta\rvert\leq 1/100k$. Given any Bohr set $B$ we can always find a regular Bohr set that closely approximates it. The following is given as Lemma 4.25 in \cite{taovu}.
\begin{lemma}
 Given a Bohr set $B_\rho(\Gamma)$ there is $\epsilon\in[1/2,1)$ such that $B_{\epsilon\rho}(\Gamma)$ is regular.
\end{lemma}

Any statement which relies on the subspace structure of Bohr sets $B$ in the previous arguments can then be replaced with an approximate version, incurring an error term dependent on the width parameter $\rho$, which can be controlled by choosing $\rho$ sufficiently small. In many cases this means $\rho$ must depend on $\alpha$, which results in the extra logarithmic factors in Theorem \ref{mainthm1}. For example, the appeal to the identity $\langle A,B\ast\beta'\rangle=\alpha\mu_G(B)$ in the proof of Lemma \ref{hbs} can be replaced by an appeal to the following lemma. 
\begin{lemma}
 If $B=B_\rho(\Gamma)$ is a regular Bohr set with rank $k$ and $B'\subset B_{\epsilon\rho}(\Gamma)$ for some $\epsilon\ll 1/k$ then for any $f:G\to\bbc$ with $\| f\|_\infty\leq 1$ we have
\[\langle f,B\ast\beta'\rangle=\bigg(\bbe_{x\in B}f(x)+O(\epsilon k)\bigg)\mu_G(B).\]
\end{lemma}
\begin{proof}
Since $\supp(B\ast\beta')=B+B'\subset B_{(1+\epsilon)\rho}(\Gamma)$, we have
\[\lvert G\rvert\langle f,B\ast\beta'\rangle=\sum_{x\in B_{(1+\epsilon)\rho}(\Gamma)}f(x)B\ast\beta'(x).\]
Furthermore, whenever $x\in B_{(1-\epsilon)\rho}(\Gamma)$, we have $B'+x\subset B$ and hence $B\ast\beta'(x)=1$. It follows that the sum above is
\[\sum_{x\in B}f(x)+\sum_{x\in B\backslash B_{(1-\epsilon)\rho}(\Gamma)}f(x)(B\ast\beta'(x)-1)+\sum_{x\in B_{(1+\epsilon)\rho}(\Gamma)\backslash B}f(x)B\ast\beta'(x).\]
By regularity, however, in each of the second sums there are $O(\epsilon k)\lvert B\rvert$ summands, and the result follows by the triangle inequality.
\end{proof}

We restrict ourselves to stating the main tools needed for the proof of Theorem \ref{mainthm1}, which closely follows the proof of Theorem \ref{mainthm2}. Making the required technical changes is a simple but lengthy matter, and is easily done by comparing the proofs given in the rest of this paper to the proofs for the $s=3$ and $\bbz/N\bbz$ case given in \cite{sanders1}, which works in this more complicated setting from the beginning.

The following are the analogues of Theorems \ref{kkthm} and \ref{csthm} (and also Proposition 4.1 and Corollary 5.2 from \cite{sanders1}) respectively. Note that the only change is that we restrict certain sets to be inside a sub-Bohr set of the main Bohr set, with width suitably restricted.

\begin{theorem}\label{int1}
Let $B_\rho(\Gamma)$ and $B'=B_{\rho'}(\Gamma')$ be regular rank $d$ Bohr sets such that $B'\subset B_{\epsilon\rho}(\Gamma)$. Suppose further that $A_1\subset B$ and $A_2,\hdots,A_{k+1}\subset B'$ all with density at least $\alpha$. Finally, suppose $\epsilon\leq c \alpha_1/d$ and $\epsilon'\leq c\alpha/d$ for some sufficiently small absolute constant $c>0$. Then either
\begin{enumerate}
 \item  there exists a regular Bohr set $B''\subset B'$ such that
\[\text{rk}(B'')\leq\text{rk}(B)+O(\alpha_1^{-1/k}\log(1/\alpha)),\]
\[\beta'(B'')\geq\left(\frac{\alpha_1}{d\log(1/\alpha)}\right)^{C(d+\alpha_1^{-1/k}\log(1/\alpha))}\]
and, for some $x\in B$, we have $\beta''(A_1+x)\geq2\alpha_1$, or
 \item there are sets $L\subset B$ and $S_1,\hdots,S_k\subset B''$ with $\lambda\gg 1$ and $\sigma_i\geq\alpha^{C\alpha_1^{-1/k}}$ for $1\leq i\leq k$ such that
\[L\ast S_1\ast\cdots\ast S_k(x)\leq\alpha_1^{-2}A_1\ast A_2\ast\cdots\ast A_{k+1}(x)\]
 for all $x\in G$, where all convolutions are taken over the measure $\beta'$. 
\end{enumerate}
\end{theorem}

\begin{theorem}\label{int2}
Let $B_\rho(\Gamma)$ and $B'$ be regular rank $d$ Bohr sets such that $B'\subset B_{\epsilon\rho}(\Gamma)$ and suppose $A,L\subset B$ and $S_1,\hdots,S_k\subset B'$. Suppose further that $\epsilon\leq c\lambda\alpha/d$ for some sufficiently small absolute constant $c>0$. Then either
 \begin{enumerate}
  \item $\langle L\ast S_1\ast\cdots\ast S_k,A\rangle_\beta\gg \lambda\sigma_1\cdots\sigma_k\alpha$, or
  \item for any integer $l$ there exists a regular Bohr set $B''\subset B'$ and an $m$ satisfying
\[m\ll \lambda^{-2-1/2l}\alpha^{-1/2l}l^2\log(1/\alpha)\log(1/\sigma_k)\]
with rank $\text{rk}(B'')\leq \text{rk}(B)+m$, density $\beta'(B'')\geq (1/dm)^{C(d+m)}$ and, for some $x\in B$, we have $\beta''(A+x)\geq\alpha(1+c\lambda)$.
 \end{enumerate}
\end{theorem}

The following is the main density increment lemma needed to prove Theorem \ref{mainthm1}, and is proved by combining Theorems \ref{int1} and \ref{int2} as in the proof of the analogous Lemma \ref{di}.
\begin{lemma}\label{di2}
Let $B_\rho(\Gamma)$ and $B'=B_{\rho'}(\Gamma')$ such that $B'\subset B_{\epsilon\rho}(\Gamma)$. Suppose further that $A_1,A_s\subset B$ and $A_2,\hdots,A_{s-1}\subset B'$ all with density at least $\alpha$. Finally, suppose that $\epsilon\leq c/d\min(\alpha_1,\alpha_s)$ and $\epsilon'\leq c\alpha/d$ for some sufficiently small absolute constant $c>0$. Then either
\begin{enumerate}
 \item $\langle A_1\ast A_2\cdots\ast A_{s-1},A_s\rangle_\beta\gg\alpha^{C\alpha_1^{-1/(s-2)}}\alpha_s$, or
\item we have a Bohr set $B''\subset B'$ such that
\[\text{rk}(B'')\leq\text{rk}(B)+O(\alpha_1^{-1/(s-2)}\log^3(1/\alpha_s)\log(1/\alpha)),\]
\[\beta'(B'')\geq\left(\frac{\alpha_1}{d\log(1/\alpha)}\right)^{C(d+\alpha_1^{-1/(s-2)}\log^3(1/\alpha_s)\log(1/\alpha))}\]
such that for some $x\in B$ and $i\in\{1,s\}$, $\beta''(A_i+x)\geq\alpha(1+c)$ for some absolute constant $c>0$.
\end{enumerate}
\end{lemma}

Theorem \ref{mainthm1} is then proved by repeated applications of Lemma \ref{di2}, following the strategy of the proof of Theorem \ref{big}.
\section*{acknowledgements}
I would to thank my PhD supervisor Professor Trevor Wooley for pointing me towards this area of research and many useful discussions. I would also like to thank the anonymous referee of this paper for their helpful suggestions and corrections.
\bibliographystyle{amsplain}
\bibliography{MyTrInvEqSan}

\providecommand{\bysame}{\leavevmode\hbox to3em{\hrulefill}\thinspace}
\providecommand{\MR}{\relax\ifhmode\unskip\space\fi MR }
\providecommand{\MRhref}[2]{%
  \href{http://www.ams.org/mathscinet-getitem?mr=#1}{#2}
}
\providecommand{\href}[2]{#2}
\begin{thebibliography}{10}

\bibitem{batemankatz}
M.~Bateman and N.~H. Katz, \emph{{New bounds on cap sets}}, arXiv:1101.5851,
  2011.

\bibitem{bourgain1}
J.~Bourgain, \emph{On triples in arithmetic progression}, Geom. Funct. Anal.
  \textbf{9} (1999), 968--984.

\bibitem{bourgain2}
\bysame, \emph{Roth's theorem on progressions revisited}, J. Anal. Math.
  \textbf{104} (2008), 155--206.

\bibitem{chang}
M.-C. Chang, \emph{A polynomial bound in {F}reiman's theorem}, Duke Math. J.
  \textbf{113} (2002), 399--419.

\bibitem{crootsisask}
E.~S. Croot and O.~Sisask, \emph{A probabilistic technique for finding
  almost-periods of convolutions}, Geom. Funct. Anal. \textbf{20} (2010),
  1367--1396.

\bibitem{dyson}
F.~J. Dyson, \emph{A theorem on the densities of sets of integers}, J. London.
  Math. Soc. \textbf{20} (1945), 8--14.

\bibitem{green}
B.~Green, \emph{Finite field models in additive combinatorics}, Surveys in
  Combinatorics 2005, Cambridge University Press, 2005, London Mathematical
  Society Lecture Note Series (No. 327).

\bibitem{heath-brown}
D.~R. Heath-Brown, \emph{Integer sets containing no arithmetic progressions},
  J. London Math. Soc. (2) \textbf{35} (1987), 385--394.

\bibitem{katzkoester}
N.~H. Katz and P.~Koester, \emph{On additive doubling and energy}, SIAM J.
  Discrete Math. \textbf{24} (2010), 1684--1693.

\bibitem{kubota}
R.~M. Kubota.

\bibitem{liuspencer}
Y.-R. Liu and C.~V. Spencer, \emph{A generalization of {R}oth's theorem in
  function fields}, Int. J. Number Theory \textbf{5} (2009), 1149--1154.

\bibitem{liuzhao}
Y.-R. Liu and X.~Zhao, \emph{A generalization of {R}oth's theorem in function
  fields}, to appear.

\bibitem{meshulam}
R.~Meshulam, \emph{On subsets of finite abelian groups with no 3-term
  arithmetic progressions}, J. Combin. Theory Ser. A \textbf{71} (1995),
  168--172.

\bibitem{roth}
K.~F. Roth, \emph{On certain sets of integers}, J. London Math. Soc.
  \textbf{28} (1953), 104--109.

\bibitem{roth2}
\bysame, \emph{On certain sets of integers {(II)}}, J. London Math. Soc.
  \textbf{29} (1954), 20--26.

\bibitem{ruzsa}
I.~Ruzsa, \emph{Solving a linear equation in a set of integers {I}}, Acta
  Arith. \textbf{65} (1993), 259--282.

\bibitem{sanders2}
T.~Sanders, \emph{On certain other sets of integers}, J. Anal. Math. (to
  appear).

\bibitem{sanders1}
\bysame, \emph{On {R}oth's theorem on progressions}, Annals of Math. (to
  appear).

\bibitem{schoen}
T.~Schoen and I.~Shkredov, \emph{{Roth's theorem in many variables}},
  arXiv:1106.1601.

\bibitem{szemeredi}
E.~Szemer\'{e}di, \emph{Integer sets containing no arithmetic progressions},
  Acta Math. Hungar. \textbf{56} (1990), 155--158.

\bibitem{taovu}
T.~Tao and V.~Vu, \emph{Additive combinatorics}, Cambridge Studies in Advanced
  Mathematics 105, Cambridge University Press, 2006.

\end{thebibliography}

\end{document}